\documentclass[a4paper,11pt]{amsart} %

\usepackage{geometry}
\usepackage{amsmath, amsthm}
\usepackage{amssymb, mathtools}
\usepackage{extarrows}

\usepackage[dvipsnames]{xcolor}

\usepackage{bigfoot}

\usepackage[shortlabels]{enumitem}

\usepackage{todonotes}

\usepackage{hyperref}

\usepackage{tikz-cd}

\usepackage{endnotes}
\usepackage{booktabs}
\usepackage{tabulary}

\usepackage{soul}
\usepackage{scalerel}

\usepackage[utf8]{inputenc}
\usepackage[greek,english]{babel}
\usepackage{lmodern}

\newtheorem{THEOREM}{Theorem}[section]

\newtheorem{Conclusion}[THEOREM]{Conclusion}

\newtheorem{Theorem}[THEOREM]{Theorem}
\newenvironment{theorem}{\begin{Theorem}}{\end{Theorem}}

\newtheorem{Lemma}[THEOREM]{Lemma}
\newenvironment{lemma}{\vskip6pt\begin{Lemma}}{\end{Lemma}\vskip6pt}

\newtheorem{Proposition}[THEOREM]{Proposition}
\newenvironment{proposition}{\vskip12pt\begin{Proposition}}{\end{Proposition}\vskip6pt}

{\newtheoremstyle{definition}
  {12pt}
  {0pt}
  {\upshape\color{black}}
  {}
  {\bfseries\color{black}}
  {.}
  {.5em}
  {}
\theoremstyle{definition}

\newtheorem{notation}[THEOREM]{Notation}
\newtheorem{definition}[THEOREM]{Definition}

\newtheorem{remark}[THEOREM]{Remark}

\newtheorem*{note-nono}{Note}
\newtheorem{note-no}{Note}

}

\newtheorem{Claim}[THEOREM]{Claim}

\newtheorem{Subclaim}[THEOREM]{Subclaim}

\newtheorem{Corollary}[THEOREM]{Corollary}
\newenvironment{corollary}{\begin{Corollary}}{\end{Corollary}}

\newtheoremstyle{myrmkstyle}
  {12pt}
  {3pt}
  {\upshape\color{blue}}
  {}
  {\bfseries\color{blue}}
  {.}
  {.5em}
  {}

  \theoremstyle{myrmkstyle}

\newtheoremstyle{myrmkstyleblk}
  {12pt}
  {0pt}
  {\upshape\color{black}}
  {}
  {\bfseries\color{black}}
  {.}
  {.5em}
  {}

  \theoremstyle{myrmkstyleblk}

\newtheoremstyle{privatermkstyle}
  {12pt}
  {3pt}
  {\upshape\color{blue}}
  {}
  {\bfseries\color{blue}}
  {}
  {0em}
  {}

  \theoremstyle{privatermkstyle}

\newtheoremstyle{privatermkstylered}
  {3pt}
  {3pt}
  {\upshape\color{red}}
  {}
  {\bfseries\color{red}}
  {}
  {0em}
  {}

  \theoremstyle{privatermkstylered}

\newenvironment{eqn}{\begin{equation*}}{\end{equation*}}
\newenvironment{itmz}{\begin{itemize}[itemsep=6pt,topsep=-3pt,leftmargin=0.15in]}{\end{itemize}}

\def\>{\rangle}
\def\<{\langle}

\newcommand{\rge}{\mathop{\mathrm{rge}}} 
\newcommand{\dom}{\mathop{\mathrm{dom}}} 
\newcommand{\Var}{\mathop{\mathrm{Var}}} 
\newcommand{\FV}{\mathop{\mathrm{FV}}} 
\renewcommand{\lim}{\mathop{\mathrm{lim}}} 
\renewcommand{\ni}{\notin}
\renewcommand{\to}{\mathop{\parbox{.5cm}{\rightarrowfill}}}

\newcommand{\lh}{\mathop{\mathrm{lh}}}
\newcommand{\On}{\mathop{\mathrm{On}}}
\newcommand{\Card}{\mathop{\mathrm{Card}}}

\newcommand{\concat}{\kern-.25pt\raise4pt\hbox{$\frown$}\kern-.25pt}
\newcommand{\on}{\upharpoonright}  

\newcommand{\Nbb}{\mathbb{N}}

\newcommand{\pret}{\sss\mathrel{\triangleleft}\sss}

\newcommand{\isom}{\cong}
\newcommand{\thinks}{\models}

\newcommand{\Vee}{\bigvee}
\newcommand{\Wedge}{\bigwedge}

\newcommand{\UA}{{\mathcal U}{\mathcal A}}
\def\sss{\hskip2pt}

\def\ssss{\hskip1pt}
\def\tupof#1#2{\<\ssss #1\sss:\sss #2\ssss\>} 
\def\tup#1{\<\ssss #1\ssss\>} 
\def\setof#1#2{\{\ssss #1\sss:\sss #2\ssss\}} 
\def\set#1{\{\ssss #1\ssss\}} 

\newcommand{\cardrule}{\hrule height.2pt}
\newcommand{\cardrulefill}{\cleaders\cardrule\hfill}
\newcommand{\cardchar}[1]{\vbox{\ialign{##\crcr
    \cardrulefill\crcr\noalign{\kern1pt\nointerlineskip}
    $\hfil\displaystyle{#1}\hfil$\crcr}}}
\newcommand{\barchar}[1]{\vbox{\ialign{##\crcr
    \cardrulefill\crcr\noalign{\kern2.5pt\nointerlineskip}
    $\hfil\displaystyle{#1}\hfil$\crcr}}}
\def\card#1{\cardchar{\cardchar{#1}}}
\newcommand{\obar}[1]{\barchar{#1}}

\newcommand{\forkindep}[2][]{%
  \mathrel{
    \mathop{
      \vcenter{
        \hbox{\oalign{\noalign{\kern-.3ex}\hfil$\vert$\hfil\cr
              \noalign{\kern-.7ex}
              $\smile$\cr\noalign{\kern-.3ex}}}
      }
    }\displaylimits^{#2}_{#1}
  }
}

\title{Scott-Karp analysis without sentences}

\author{Andreas Brunner, Charles Morgan and Darllan Concie\c{c}\~ao Pinto}

\address{Departamento de Matem\'atica, Universidade Federal de Bahia,
  Avenida Milton Santos s/n, S/N, Ondina, Salvador, Bahia, Brazil. CEP: 40170-110.}
\email{andreas@dcc.ufba.br}
\email{charlesmorgan@ufba.br}
\email{darllan@ufba.br}

\begin{document}

\begin{abstract} Scott and Karp gave an analysis which provides a level-by-level equivalence between
   global similarity between two structures and local commonality in terms of sharing particular invariants.
  Scott and Karp's local invariants were certain infinitary formulae. 

  We give a more abstract version of the local side of Scott-Karp analysis which avoids the use of infinitary languages.
  We show the resulting hierarchies are still provide desired equivalences in the classical setting.

  Moreover, the abstract nature of our analysis, as we show, 
  makes it suitable to provide local level-by-level equivalents to Hjorth's much more general version of global similarity
  in the context of topological group actions on topological groups.

  We, furthermore, provide, analogously to the classical work of Ehrenfeucht and Fra\"iss\'e, novel game theoretic equivalents
  for Hjorth's global similarity relations and some natural variants.
\end{abstract}

\maketitle

\section*{Introduction}\label{intro}

Cantor celebratedly showed (\cite{Cantor}) any two countable dense linear orders without
endpoints are isomorphic. The \emph{back and forth} method appears\footnote{See \cite{Plotkin},
\cite{Silver}.}  to have been introduced independently by Huntingdon
(\cite{Huntingdon}) and Hausdorff (\cite{Hausdorff07},
\cite{Hausdorff}) to give proofs of this result.

We sketch such a proof in order to give
the reader a taste of the method.
Take two such orders, $\mathcal O_0$ and $\mathcal O_1$, and enumerate
each in order type $\omega$.  Build a sequence of partial isomorpisms
between them by induction.  Start with the empty function. For the
inductive step, suppose the partial isomorphism $f$ has been built so
far.  Take the first element, say $x$, in the enumeration of $\mathcal
O_0$ not in $\dom(f)$. Find an element $y$ of $\mathcal O_1$ which
lies in the same position in $\mathcal O_1$ with respect to the
elements of $\rge(f)$ in $\mathcal O_1$ as $x$ does to the elements of
$\dom(f)$ in $\mathcal O_1$ and set $g = f\cup\set{(x,y)}$. (This is
  the ``forth'' manoeuvre.) Then, take the first element, $z$, in the
  enumeration of $\mathcal O_1$ not in $\rge(g)$.  Find an element $w$
  of $\mathcal O_0$ which lies in the same position in $\mathcal O_0$
  with respect to the elements of $\dom(g)$ in $\mathcal O_0$ as $z$
  does to the elements of $\rge(g)$ in $\mathcal O_1$ and set $h =
  g\cup\set{(w,z)}$. (This is the ``back'' manoeuvre.)  Taking the
    union of the chain of partial isomorphisms generated in this way gives an order
    preserving bijection as required.

As employed in this way to prove isomorphism between two structures, the method
has proven very fruitful in model theory. Poizat (\cite{Poizat-en},\cite{Poizat-fr}), Chapter 6, is 
a rich source of introductory examples, from the theories of algebraically, differentially or real closed fields, Boolean algebras,
ultrametric spaces and modules,\footnote{Some material was added to the
original version of this chapter in \cite{Poizat-fr} in the course of its translation to \cite{Poizat-en}.}
and Marker (\cite{Marker}), Theorem (2.4.2), for example,
proves the uniqueness up to isomorphism of the countable random graph.

However, when applied to two non-isomorphic structures (for the same language)
the technique can still supply useful and interesting information. 
One instantiation of this phenomenon is what Marker (\cite{Marker}) refers to as \emph{Scott-Karp analysis}, (manifested in
\cite{Scott} and \cite{Karp}).

Suppose one is interested in comparing two pairs each consisting of a
model and a sequence of elements from the model, say $(M,\bar{a})$ and $(N,\bar{b})$, where $M$, $N$ are $L$-structures
for some language $L$ and $\lh(\bar{a})=\lh(\bar{b})$.

The basic idea of Scott-Karp analysis is to provide two ordinal-indexed hierarchies within each of which one can compare
such pairs. One hierarchy should measure global similarity,
whilst the other should measure the extent to which pairs have specific (single) invariants in common. This type of analysis
is useful if one can prove the two hierarchies match: that particular degrees of global similarity correspond exactly
to particular invariants being in common. If one can do so one can read off global information about a pair from local information.

In Scott and Karp's original analysis the level of global similarity between two
such pairs was determined by whether exactly the same formulae up to a certain degree
of complexity, measured in terms of an ordinal-valued rank derived from the quantifiers used in the formulae,
were satisfied by both of the two pairs. This is very closely related to a hierarchy of
relations defined via the back and forth method, which one should also think of as being concerned with global information
about the pairs. (The rank and a useful variant are defined below.)

The elements of the hierarchy of invariants in Scott and Karp's analysis were single infinitary formulae $\phi^\alpha_{M,\bar{a}}$,
corresponding to a pair $(M,\bar{a})$ and an ordinal $\alpha$. (These formulae are also touched on below, but, as
indicated by the title, do not play a central r\^ole here.)

Karp's theorem, which generalized a theorem of Fra\"iss\'e (\cite{Fra53}) for finite $\alpha$,
is that the two hierarchies do in fact match up.

Much later, Hjorth (\cite{Hjorth}, \cite{Drucker})
gave a version of the first, global, leg of this Scott-Karp analysis,
in the sense of the hierarchy of relations defined via the back and forth method,
in the much more general context of topological group actions on topological spaces.
This should still be understood as paying attention to similarity at the global scale.
We discuss this in more detail in \S{}\ref{Hjorth_Karp_thm}  below.

  In this paper we give a more abstract version of the second, local, leg of Scott-Karp analysis.
  We show this alternative approach still allows one to match up global and local hierarchies.
  This allows us to give alternative proofs in the original context of Scott and Karp's work.

    Specifically, given a signature $\tau$ and $L$ either the first
  order language or the infinitary language $L_{\infty\lambda}$, to
  each $L$-structure $M$ and sequence $\bar{a}$ of elements and
  ordinal $\alpha$ we define functions $F^{\alpha}_{M,\bar{a}}$, and
  $H^{\alpha}_{M,\bar{a}}$, and for each $L$-structure $N$ a function
  $G^{\alpha}_{M,N\bar{a}}$.

  The definition of the $F^{\alpha}_{M,\bar{a}}$ is recursive and so
  somewhat complex to understand directly.  The definition of the
  $H^{\alpha}_{M,\bar{a}}$ is more straightforward, however they are
  class functions.  Nevertheless, we show that given $(M,\bar{a})$,
  $(N,\bar{b})$ and $\alpha$ we have $F^{\alpha}_{M,\bar{a}} =
  F^{\alpha}_{N,\bar{b}}$ if and only if $H^{\alpha}_{M,\bar{a}} =
  H^{\alpha}_{N,\bar{b}}$.  Thus it is a matter of taste whether to
  focus the $F^{\alpha}_{M,\bar{a}}$ or the $H^{\alpha}_{M,\bar{a}}$.

  We then show that $F^{\alpha}_{M,\bar{a}} = F^{\alpha}_{N,\bar{b}}$
  if and only if $(M,\bar{a}) \sim_\alpha (N,\bar{b})$, where the
  latter means there is a back and forth system between $(M,\bar{a})$
  and $(N,\bar{b})$.

  We follow this by defining (infinitary) formulae $\phi^\alpha_{M,\bar{a}}$ and showing
  $N\thinks \phi^{\alpha}_{M,\bar{a}}(\bar{b})$ if and only if $F^{\alpha}_{M,\bar{a}} = F^{\alpha}_{N,\bar{b}}$.
  This allows us to give a version of Karp's theorem, $M\equiv_\alpha N$ if and only if $M\sim_\alpha N$
  if and only if $F^{\alpha}_{M,\emptyset} = F^{\alpha}_{N,\emptyset}$.
  
  However, our approach is also applicable to topological group actions on topological spaces, allowing us
  to prove new results there analogous to those of Scott and Karp classically, where the lack of languages
  vitiates the usual Scott-Karp analysis.

  Since Ehrenfeucht-Fra\"iss\'e games are closely tied to (and many
  authors assert give an alternative account for) the global leg of
  Scott-Karp analysis, it seemed worthwile to also give in the paper
  (\S{}\ref{Hjorth_intro_games}) a game theoretic analysis of Hjorth's
  hierarchy. We give equivalents to three notions of comparability
  of pairs consisting of a point in the space and an open set in the group
  in terms of winning strategies for player II in certain games we define.

  We then show results in this context similar to our earlier first order model theoretic ones,
  for example that for certain functions $H^\alpha_{x,U}$ we have
  $H^\alpha_{x,U} = H^\alpha_{y,V}$ if and only if $(x,U)\sim_\alpha (y,V)$ if and only if (for $\alpha>0$)
  player II has a winning strategy for ${\mathcal G}^s(x,U,y,V,\alpha)$.

\section{Preliminaries}\label{prelims}

\begin{notation} Most of our notation is standard. We use a handful of standard set theoretic conventions. We write
$\On$ for the class of ordinals. Given a set $X$ and an ordinal $\alpha$
we denote by $\card{X}$ the cardinality of $X$ and $^{\alpha}X$ the set of sequence of length $\alpha$ from $X$, equivalently
the set of functions from $\alpha$ to $X$, while $^{<\alpha}X = \bigcup_{\beta<\alpha} {^{\beta}X}$. 

We follow the typical model theoretic convention and denote concatenation by juxtaposition.
If $\bar{a}$, $\bar{b}$ are tuples from a structure $M$ and $c\in M$,
  we write $\bar{a}\bar{b}$ for the concatenation $\bar{a}\concat\bar{b}$,
  and $\bar{a}c$ for the concatenation $\bar{a}\concat c$. More
  explicity, if $\bar{a}=\tupof{a_i}{i<\alpha}\in {^{\alpha}M}$,
  $\bar{b}=\tupof{b_j}{j<\beta}\in {^{\beta}M}$ and $c\in M$, write
  $\bar{a}\bar{b}$ for the sequence $\bar{d}=\tupof{d_i}{i<\alpha +
    \beta}$, where for $i<\alpha$ we have $d_i=a_i$ and for $i =
  \alpha +j$ with $j<\beta$ we have $d_i = b_j$, and write $\bar{a}c$
  for the sequence $\bar{e}=\tupof{e_i}{i<\alpha + 1}$ where for
  $i<\alpha$ we have $e_i=a_i$ and $e_\alpha = c$.
\end{notation}
  
\begin{definition}A \emph{signature} is a collection of function, relation and constant symbols. More formally,
  $\tau$ is a signature if it is a quadruple $\tup{\mathcal F,\mathcal R,\mathcal C,\tau'}$, where
  $\mathcal F$, $\mathcal R$ and $\mathcal C$ are pairwise disjoint sets and $\tau':\mathcal F\cup\mathcal R\longrightarrow \Nbb$.
\end{definition}

Rothmaler \cite{Rothmaler}, for example, gives an explicit account of the construction of a first order language over
a signature.

Fix for the remainder of the paper  a signature $\tau$, $=\tup{\mathcal F,\mathcal R,\mathcal C,\tau'}$.

We need a concrete definition of infinitary languages.

 We adopt (a mild adaptation of)
V\"a\"an\"anen's definition (\cite{Vaananen}), Definition (9.12) of
the infinitary language $L^{\mu}_{\kappa\lambda}$.  (\cite{Vaananen} does
\emph{not} specify that $\lambda$ should be regular.)

\begin{definition}\label{defn_L^mu_kappa_lambda}
  Let $\mu$, $\lambda$ and $\kappa$ be cardinals.
  The set $L^{\mu}_{\kappa\lambda}$ is defined by
  \begin{enumerate}
  \item For $\alpha<\lambda$ the variable $v_\alpha$  is a term. Each constant symbol $c\in  \mathcal C$ is a term.
    If $f$ is an $n$-ary
    function symbol and $t_0$, \dots, $t_{n-1}$ are terms then so is $f(t_0,\dots,t_{n-1})$.
\item If $t_0$ and $t_1$ are terms, then $t_0=t_1$ is an $L^{\mu}_{\kappa\lambda}$-formula.
\item If $t_0$,\dots, $t_{n-1}$ are terms and $R$ is an $n$-ary relation symbol then\break
   $R(t_0,\dots t_{n-1})$ is an $L^{\mu}_{\kappa\lambda}$-formula.
\item If $\phi$ is an $L^{\mu}_{\kappa\lambda}$-formula, so is $\neg \phi$.
\item If $\Phi$ is a set of $L^{\mu}_{\kappa\lambda}$-formulae of size less than $\kappa$
  with a fixed set $V$ of free variables and $\card{V} < \lambda$,
 $\Wedge\Phi$ is an $L^{\mu}_{\kappa\lambda}$-formula.
\item If $\Phi$ is a set of $L^{\mu}_{\kappa\lambda}$-formulae of size less than $\kappa$
  with a fixed set $V$ of free variables and $\card{V} < \lambda$,
  $\Vee\Phi$ is an $L^{\mu}_{\kappa\lambda}$-formula.
\item If $\phi$ is an $L^{\mu}_{\kappa\lambda}$-formula,
  $V \in [\setof{v_\alpha }{\alpha < \lambda}]^{<mu}$ and $\card{V}<\lambda$,
   $\forall V\sss\phi$ is an $L^{\mu}_{\kappa\lambda}$-formula.
\item If $\phi$ is an $L^{\mu}_{\kappa\lambda}$-formula,
  $V \in [\setof{v_\alpha }{\alpha < \lambda}]^{<\mu}$ and $\card{V}<\lambda$,
  $\exists V\sss\phi$ is an $L^{\mu}_{\kappa\lambda}$-formula.
  \end{enumerate}  
\end{definition}

Thus $\lambda$ bounds the size of the set of free variables of a formula,
$\kappa$ bounds the size of the infintary conjuctions and disjunctions, and
$\mu$ bounds the size of the sets of variables over which one may quantify.

\begin{definition}\label{defn_L_infty_lambda} For $\mu$, $\lambda$ and $\kappa$
  cardinals we set $L_{\kappa\lambda} = L^{\lambda}_{\kappa\lambda}$
  and $\bar{L}_{\kappa\lambda} = L^{2}_{\kappa\lambda}$.

  For $\mu$ and $\lambda$ cardinals we set
  $L^{\mu}_{\infty\lambda} = \bigcup_{\kappa\in \Card} L^{\mu}_{\kappa\lambda}$.
  We also set
  $L_{\infty\lambda} = L^{\lambda}_{\infty\lambda} = \bigcup_{\kappa\in \Card} L_{\kappa\lambda}$
  and
  $\bar{L}_{\infty\lambda} = L^{2}_{\infty\lambda} = \bigcup_{\kappa\in \Card}
  \bar{L}_{\kappa\lambda}$
\end{definition}

So, for example, $\bar{L}_{\omega\omega}$ is first order logic over $\tau$.

  We refer the reader to \cite{Vaananen}, \cite{Dickmann-book} and
  \cite{Dickmann} for more on $L^{\mu}_{\kappa\lambda}$.
  \cite{Vaananen} gives a formal version of this definition of $L_{\infty\lambda}$ as a class
  with a model of ZFC.

  \begin{definition}\label{defn_Var_FV} Let $\phi$ be an $L^{\mu}_{\kappa\lambda}$-formula. Then $\Var(\phi)$ is the set of variables occuring in $\phi$ and
    $\FV(\phi)$ is the set of free variables occuring in $\phi$.
\end{definition}

\section{Equivalence relations for tuples from two structures}\label{equivalence} 

There is a substantial body of work in the literature on sequences of ``back-and-forth'' relations.
See \cite{Harrison-Trainor}, \cite{Marker}, \cite{Vaananen}.  
The elements of the domains of these relations are pairs consisting of a structure and a tuple from it.
The sequences are defined by induction and aim to measure
the degree of similarity amongst the pairs. There is
quite a variety of such sequences of relations. The principal
distinctions are whether the relations in the sequence are symmetric
or asymmetric and whether the `extensions' at successor steps are by
singletons or by tuples -- the extensions in the back-and-forth argument given in the introduction, for example,
are by singletons. Further variations come from different
starting points for the inductive definitions.

Here we work with the version more usual in the analysis of structures and tuples from them
appearing in classical first order model theory, with symmetric relations.

As far as the size of the extensions go, we have our cake and
eat it. We work with languages over the signature $\tau$ in a way that
allows us to treat classical first order logic ($\bar{L}_{\omega\omega}$) and the
infinitary logic $L_{\infty\lambda}$, for $\lambda$ a cardinal,
simultaneously. Our key short-cut is to consider only languages, $L$,
parametrized by cardinals, $\lambda$ and $\mu$, where 
$\lambda$ is a strict upper bound on the size of the conjunctions
and disjunctions allowed in the formation of $L$-formulae and
$\mu$ is a strict upper bound on the size of tuples
additionable in back-and-forth steps.

The only pairs of which we permit consideration are $(2,\omega)$, corresponding to one-point
extensions, and arbitrary pairs of the form $(\lambda,\lambda)$. This
formalism is perhaps a little convoluted, however, it does allow us
to give unified proofs of our results. We stress, again, we are \emph{very} far from allowing arbitrary
cardinals for the values of $(\mu,\lambda)$. 

Accordingly, we now fix for the rest of the paper such a qualifying pair $(\mu,\lambda)$.

\begin{definition}\label{UA_fmlae} Let $\phi \in L$ be atomic. We say $\phi$ is \emph{unnested} if
  it is of one of the forms
  \[ v_i=v_j,\sss\sss v_i = c,\sss\sss R(v_{i_0},\dots,v_{i_{n-1}})\hbox{ or }v_{i_n} = f(v_{i_0},\dots,v_{i_{n-1}}),\] where the
  $v_i$ are variables, $c$ is a constant symbol, $R$ is an $n$-ary relation symbol and $f$ is an $n$-ary function symbol.
  Write ${\UA}$ for the collection of
  unnested atomic formulae. Note if $\phi\in \UA$ then $\Var(\phi) = \FV(\phi)$ and is finite.
\end{definition}

\begin{definition}\label{defn_unnested} Let $\phi \in L$. We say $\phi$ is \emph{unnested} if
  all of its atomic subformulae are unnested. Write  ${\mathcal U}$ for the collection of unnested formulae.
\end{definition}
  
\begin{remark} If $L$ is relational every formula is unnested. 
\end{remark}

\begin{definition}\label{defn_pp_fmla} A formula $\phi \in L$ is \emph{positive primitive}
  ~(a \emph{pp formula}), if it is made up from unnested atomic formulae using only conjunction and existential quantification.
\end{definition}

Our next definition is similar to that appearing in \cite{Marker}, p57.

\begin{definition}   Let $Y$ be $\UA$.\footnote{Other choices seen in the literature for
  similar definitions are all atomic formulae, all quantifier free
  formulae, all quantifier free formulae which in which negation does
  not occur and all existential formulae. \cite{Marker}, p57, uses all
  atomic formulae rather than the unnested ones as an initial step in
  the inductive definition.}  Let $M$ and $N$ be $L$-structures,
  $\bar{a}\in {^{<\lambda}M}$ and $\bar{b}\in {^{\lh(\bar{a})}N}$. Let
  $h:\bar{a}\longrightarrow \bar{b}$ be given by $h(a_i)=b_i$ for all
  $i<\lh(\bar{a})$.

  Let $(M,\bar{a})\sim_{0} (N,\bar{b})$ if for all formulae
  $\psi(\bar{v})\in Y$ with $\lh(\bar{v})=\lh(\bar{a})$ we have
  $M\thinks \psi(\bar{a})$ if and only if $N\thinks\psi(\bar{b})$. In
  this case we say $h$ is a \emph{partial isomorphism}.

  We let $Q_0$ be the set of all such partial isomorphisms (for all pairs $(M,\bar{a})$, $(N,\bar{b})$). 
 
  For $0<\alpha$ let $(M,\bar{a})\sim_{\alpha} (N,\bar{b})$ if for all $\beta<\alpha$ the following two conditions hold
  \begin{eqn}
    \begin{split} & \hbox{for all }\bar{c}\in {^{<\mu}M}\hbox{ there is some }\bar{d}\in {^{\lh(\bar{c})}N}
                      \hbox{ such that } (M,\bar{a}\bar{c}) \sim_\beta (N,\bar{b}\bar{d}) \\
                 & \hbox{for all }\bar{d}\in {^{<\mu}N}\hbox{ there is some }\bar{c}\in {^{\lh(\bar{d})}M}
                      \hbox{ such that } (M,\bar{a}\bar{c}) \sim_\beta (N,\bar{b}\bar{d}).
  \end{split}\end{eqn}

Equivalently, for all $\alpha$ we have $(M,\bar{a})\sim_{\alpha+1} (N,\bar{b})$ if 
the following two conditions hold
  \begin{eqn}
    \begin{split} & \hbox{for all }\bar{c}\in {^{<\mu}M}\hbox{ there is some }\bar{d}\in {^{\lh(\bar{c})}N} \hbox{ such that } (M,\bar{a}\bar{c}) \sim_\alpha (N,\bar{b}\bar{d}) \\
     & \hbox{for all }\bar{d}\in {^{<\mu}N}\hbox{ there is some }\bar{c}\in {^{\lh(\bar{d})}M} \hbox{ such that } (M,\bar{a}\bar{c}) \sim_\alpha (N,\bar{b}\bar{d}).
  \end{split}\end{eqn}and for all limit $\alpha\in \On$ we have $(M,\bar{a})\sim_{\alpha} (N,\bar{b})$ if and only if for all $\beta<\alpha$ we have
  $(M,\bar{a})\sim_{\beta} (N,\bar{b})$.

  We say $h\in Q_{\alpha}$ if $(M,\bar{a})\sim_{\alpha} (N,\bar{b})$.

  Write $M\sim_\alpha N$ if $(M,\emptyset)\sim_{\alpha} (N,\emptyset)$.
 \end{definition}

\section{Scott-Karp analysis without formulae}\label{Scott_analysis_wo_fmlae} 

Let $M$, $N$ be $L$-structures and $\bar{a}\in {^{<\lambda}M}$, $\bar{b}\in {^{<\lambda}N}$ with
$\lh(\bar{a})=\lh(\bar{b})$. In this section we
define functions $F^\alpha_{M,\bar{a}}$, $G^\alpha_{M,N\bar{a}}$ and $H^\alpha_{M,  \bar{a}}$, and
show they contain all of the information needed to
prove Karp-style results (Proposition (\ref{prop_funs_agree_equiv_sim_classical_G}) and
Proposition (\ref{prop_funs_agree_equiv_sim_classical_FH})) showing equivalences between 
$\sim_\alpha$ holding between pairs of pairs and equality of the corresponding functions.
Thus one does not actually need to assemble information concerning
$(M,\bar{a})$ into an infinitary sentence, as Karp did, in order to prove such theorems.

\vskip12pt

\begin{definition} Let $Y =  \emph{UA} $ be as in the previous section.\end{definition}

\vskip12pt

\begin{definition} For $M$ an $L$-structure and $\bar{a}\in {^{<\lambda}M}$ define the function $I_{M,\bar{a}}$ as follows.
  Set \begin{eqn}\begin{split}
      \dom(I_{M,\bar{a}}) =  \setof{(\psi,f)}{\psi \in Y\sss\sss\&\sss\sss \exists \ssss n & \in \omega \sss\sss \Var(\psi)= \set{v_0,\dots,v_{n-1}} \\
         &\sss\sss\&\sss\sss
    f:N\longrightarrow \lh(\bar{a})} ,
    \end{split}\end{eqn}and for $(\psi,f)\in\dom(I_{M,\bar{a}})$ set\begin{eqn}
   I_{M,\bar{a}}(\psi,f)  = \begin{dcases}\quad
            1 \;\; \hbox{ if } M\thinks \psi(a_{f(0)},a_{f(1)},\dots a_{f(n-1)})) \\
            \quad 0 \;\; \hbox{ otherwise.}
           \end{dcases}
\end{eqn}
\end{definition}

\begin{definition}\label{defn_fns} Let $M$, $N$ be $L$-structures and $\bar{a}\in {^{<\lambda}M}$.
  We inductively define functions $F^\alpha_{M,\bar{a}}$, $G^{\alpha}_{M,N,\bar{a}}$ and $H^\alpha_{M,\bar{a}}$.
  For $\beta<\alpha$ we will have
  \[ F^{\beta}_{M,\bar{a}} \subseteq F^{\alpha}_{M,\bar{a}} ,\sss\sss H^{\beta}_{M,\bar{a}} \subseteq H^{\alpha}_{M,\bar{a}}  \hbox{ and }
  G^{\beta}_{M,N,\bar{a}}\subseteq G^{\alpha}_{M,N,\bar{a}}  .\]

  Let
  \[ F^0_{M,\bar{a}} = G^0_{M,N,\bar{a}} = H^0_{M,\bar{a}} = I_{M,\bar{a}}. \]

  Elements of the domain of $F^{\alpha}_{M,\bar{a}}$  will either be elements of $\dom(I_{M,\bar{a}})$ or a pair whose second entry is an ordinal less than $\alpha$.
  Element of the domains of the $G^\alpha_{M,N,\bar{a}}$ and $H^{\alpha}_{M,\bar{a}}$ will either
  be elements of $\dom(I_{M,\bar{a}})$ or triples whose last entry is an ordinal less than $\alpha$.
  For $0<\beta<\alpha$ we will have
    \[ F^{\alpha}_{M,\bar{a}}\on \beta = F^{\beta}_{M,\bar{a}},\sss\sss H^{\alpha}_{M,\bar{a}}\on \beta = H^{\beta}_{M,\bar{a}} \hbox{ and }  G^{\alpha}_{M,N,\bar{a}}\on \beta = G^{\beta}_{M,N,\bar{a}} .\]
    Thus, in each case, for limit $\alpha$ the functions are completely determined by the definitions for $\beta<\alpha$. Furthermore, for $\alpha+1$ and tuples
    in their domain with last element $\beta<\alpha$ the functions are determined by the definitions for $\alpha$.

    Lastly, we treat for $\alpha+1$ tuples in the functions' domains with final element $\alpha$.

The triple $(K,\bar{s}\bar{t},\alpha)$ is an element of $\dom(H^{\alpha+1}_{M,\bar{a}})$ if
$K$ is an $L$-structure, $\bar{s}\in {^{\lh(\bar{a})}N}$ and $\bar{t}\in {^{<\mu}N}$.

If  $(K,\bar{s}\bar{t},\alpha)$ is an element of $\dom(H^{\alpha+1}_{M,\bar{a}})$ we set
\begin{eqn}
   H^{\alpha+1}_{M,\bar{a}}(K,\bar{s}\bar{t},\alpha)  = \begin{dcases}\quad
            1 \;\; \hbox{ if there is some } \bar{c}\in {^{\lh(\bar{t})}M}\hbox{ such that } H^\alpha_{M,\bar{a}\bar{c}} = H^\alpha_{K,\bar{s}\bar{t}} \\
            \quad 0 \;\; \hbox{ otherwise.}
           \end{dcases}
\end{eqn}

The triple $(K,\bar{s}\bar{t},\alpha)$ is an element of $\dom(G^{\alpha+1}_{M,N,\bar{a}})$
if and only if $K\in \set{M,N}$, $\bar{s}\in {^{\lh(\bar{a})}K}$ and $\bar{t}\in {^{<\mu}K}$.

For $(K,\bar{s}\bar{t},\alpha) \in \dom(G^{\alpha+1}_{M,N,\bar{a}})$, if $K=M$ set $J=N$ and if $K=N$ set $J=M$, and set
\begin{eqn}
   G^{\alpha+1}_{M,N,\bar{a}}(K,\bar{s}\bar{t},\alpha)  = \begin{dcases}\;\;
     1 \; \hbox{ if there is some } \bar{c}\in {^{\lh(\bar{t})}M}\hbox{ such that }
     G^\alpha_{M,N,\bar{a}\bar{c}} = G^\alpha_{K,J,\bar{s}\bar{t}} \\
            \;\; 0 \;\hbox{ otherwise.}
           \end{dcases}
\end{eqn}

A pair $(x,\alpha)$ is an element of $\dom(F^{\alpha+1}_{M,\bar{a}})$ if and only if there is $L$-structure $K$, $\bar{s}\in {^{\lh(\bar{a})}K}$
and $\bar{t}\in {^{<\mu}N}$ such that $x = F^{\alpha}_{K,\bar{s}\bar{t}}$.

 If  $K$ is an $L$-structure, $\bar{s}\in {^{\lh(\bar{a})}K}$ and $\bar{t}\in {^{<\mu}N}$ then
\begin{eqn}
   F^{\alpha+1}_{M,\bar{a}}(F^{\alpha}_{K,\bar{s}\bar{t}},\alpha)  = \begin{dcases}\quad
            1 \;\; \hbox{ if there is some } \bar{c}\in {^{\lh(\bar{t})}M}\hbox{ such that } F^\alpha_{M,\bar{a}\bar{c}} = F^\alpha_{K,\bar{s}\bar{t}} \\
            \quad 0 \;\; \hbox{ otherwise.}
           \end{dcases}
\end{eqn}

\end{definition}

For $0<\alpha$, clearly, $\dom(H^\alpha_{M,\bar{a}})$ is a proper class, while $F^\alpha_{M,\bar{a}}$ is defined recursively.
The following result shows we can pick our poison: whatever we can prove using one collection of functions we can also
prove with the other.

\begin{lemma}\label{reln_Fs_and_Hs}   Let $M$ be an $L$-structure, $\bar{a}\in {^{<\lambda}|M|}$,  and $\alpha\in\On$.
  Let $K$ be an $L$-structure, $\bar{s}\in {^{\lh(\bar{a})}|K|}$ and $\bar{t}\in {^{<\mu}N}$. Let $\beta<\alpha$.
  Then
                \[ F^{\alpha}_{M,\bar{a}}(F^\beta_{K,\bar{s}\bar{t}},\beta) = H^{\alpha}_{M,\bar{a}}(K,\bar{s}\bar{t},\beta) .\]
  Moreover, if further $N$ is an $L$-structure and $\bar{b}\in {^{\lh(\bar{a})}|N|}$ 
  \[ F^{\alpha}_{M,\bar{a}} = F^{\alpha}_{N,\bar{b}} \hbox{ if and only if } H^{\alpha}_{M,\bar{a}} = H^{\alpha}_{N,\bar{b}} .\]
\end{lemma}

\begin{proof} We work by induction on $\alpha$. The limit $\alpha$ case is immediate by induction.
  For successor $\alpha$ with $\beta+1<\alpha$ the result is also immediate by induction. 
  For sucessor $\alpha+1$ we have $F^{\alpha+1}_{M,\bar{a}}(F^\alpha_{K,\bar{s}\bar{t}},\alpha) = 1$
  \begin{eqn}
    \begin{split}
       & \hbox{  if and only if there is some $\bar{c}\in {^{\lh(\bar{t})}M}$ such that } F^{\alpha}_{K,\bar{s}\bar{t}} = F^{\alpha}_{M,\bar{a}\bar{c}} \\
& \hbox{  if and only if there is some $\bar{c}\in {^{\lh(\bar{t})}M}$ such that  } H^{\alpha}_{K,\bar{s}\bar{t}} = H^{\alpha}_{M,\bar{a}\bar{c}}\\
& \hbox{ if and only if }H^{\alpha+1}_{M,\bar{a}}(K,\bar{s}\bar{t},\alpha)  =1.
    \end{split}
  \end{eqn}where the middle equality follows from by induction from the ``Moreover'' part of the conclusion for $\alpha$.
 The ``Moreover'' part of the conclusion for $\alpha+1$ is now immediate.  
\end{proof}

The next lemma shows what happens when we apply the functions to domain elements derived from the
function parameters themselves.

\begin{lemma}\label{lemma_self_classical_H} For each $0\le \beta<\alpha$ if $\bar{c}\in {^{<\mu}M}$ then \[
 F^\alpha_{M,\bar{a}}(F^\beta_{M,\bar{a}\bar{c}},\beta) =  G^{\alpha}_{M,N,\bar{a}}( M,\bar{a}\bar{c},\beta) = H^\alpha_{M,\bar{a}}(M,\bar{a}\bar{c},\beta) =1 .\]
\end{lemma}

\begin{proof} By the inductive definitions of the functions,
  it suffices to show the result for $\alpha=\beta+1$. For the $H$s, by the definition of $H^{\beta+1}_{M,\bar{a}}$ we have $H^{\beta+1}_{M,\bar{a}}(M,\bar{a}\bar{c},\beta) =1 $
  if and only if there is some $\bar{c}'\in {^{\lh(\bar{c})}M}$ such that $H^\beta_{M,\bar{a}\bar{c}'} = H^\beta_{M,\bar{a}\bar{c}}$. However, taking $\bar{c}'=\bar{c}$
  shows the latter is always true. The proof for the $G$s is identical. The result for the $F$s can be shown similarly, or follows immediately from Lemma (\ref{reln_Fs_and_Hs}). 
\end{proof}

We can now give Karp-style results, a weaker version for the $G$s and a stronger one for the $F$s and $H$s.

\begin{proposition}\label{prop_funs_agree_equiv_sim_classical_G} Let $M$, $N$ be $L$-structures,
  $\bar{a}\in {^{<\lambda}M}$ and $\bar{b}\in {^{\lh(\bar{a})}N}$.
  Let $\alpha\in \On$. Then $G^\alpha_{M,N,\bar{a}} = G^\alpha_{N,M,\bar{b}}$ if and only if
  $(M,\bar{a})\sim_\alpha (N,\bar{b})$.
\end{proposition}

\begin{proposition}\label{prop_funs_agree_equiv_sim_classical_FH} Let $M$ be an $L$-structure,
  $\bar{a}\in {^{<\lambda}M}$ and
  $\alpha\in \On$. For any  $L$-structure $N$ and $\bar{b}\in {^{\lh(\bar{a})}N}$ we have
  $F^{\alpha}_{M,\bar{a}} = F^{\alpha}_{N,\bar{b}}$ if and only if
  $H^\alpha_{M,\bar{a}} = H^\alpha_{N,\bar{b}}$ if and only if
  $(M,\bar{a})\sim_\alpha (N,\bar{b})$.
\end{proposition}

The propositions have very similar proofs. Both are by induction.
Consequently, we give a unified proof.

\begin{proof}  The proof is by induction. In the statement of Proposition (\ref{prop_funs_agree_equiv_sim_classical_G})
  we are given a pair $(N,\bar{b})$. For the proof of Proposition (\ref{prop_funs_agree_equiv_sim_classical_FH})
  let us fix a pair $(N,\bar{b})$ with $N$ an $L$-structure and  $\bar{b}\in {^{\lh(\bar{a})}N}$.

  As we have already shown   $F^{\alpha}_{M,\bar{a}} = F^{\alpha}_{N,\bar{b}}$ if and only if
  $H^\alpha_{M,\bar{a}} = H^\alpha_{N,\bar{b}}$ in Lemma (\ref{reln_Fs_and_Hs}), we only need to
  show $H^\alpha_{M,\bar{a}} = H^\alpha_{N,\bar{b}}$ if and only if
  $(M,\bar{a})\sim_\alpha (N,\bar{b})$ in order to prove Proposition (\ref{prop_funs_agree_equiv_sim_classical_FH}).
  It will become clear we could equally well have chosen to prove 
  $F^\alpha_{M,\bar{a}} = F^\alpha_{N,\bar{b}}$ if and only if
  $(M,\bar{a})\sim_\alpha (N,\bar{b})$.

  In order to run a unified proof, let $E^\alpha_{M,\bar{a}}$ be either 
  $H^\alpha_{M,\bar{a}}$ or $G^\alpha_{M,N,\bar{a}}$ and $E^\alpha_{N,\bar{b}}$ be
  $H^\alpha_{N,\bar{b}}$ or $G^\alpha_{N,M,\bar{b}}$, respectively.

  $\alpha=0$. For all $\psi\in Y$ we have $E^\alpha_{M,\bar{a}}(\psi) = E^\alpha_{N,\bar{b}}(\psi)$ if and only if
  $M\thinks \psi(\bar{a}) \Longleftrightarrow N\thinks \psi(\bar{b})$, so
  $E^\alpha_{M,\bar{a}} = E^\alpha_{N,\bar{b}}$  if and only if  $(M,\bar{a})\sim_0 (N,\bar{b})$.

  Limit $\alpha$. $E^\alpha_{M,\bar{a}} = E^\alpha_{N,\bar{b}}$ if and
  only if for all $\beta< \alpha$ we have $E^\beta_{M,\bar{a}} =
  E^\beta_{N,\bar{b}}$. By induction, for each $\beta<\alpha$, we have
  $E^\beta_{M,\bar{a}} = E^\beta_{N,\bar{b}}$ if and only if
  $(M,\bar{a})\sim_\beta (N,\bar{b})$. Thus for all $\beta< \alpha$ we
  have $E^\beta_{M,\bar{a}} = E^\beta_{N,\bar{b}}$ if and only if for
  all $\beta<\alpha$ we have $(M,\bar{a})\sim_\beta
  (N,\bar{b})$. However, for all $\beta<\alpha$ we have
  $(M,\bar{a})\sim_\beta (N,\bar{b})$ if and only if
  $(M,\bar{a})\sim_\alpha (N,\bar{b})$.

  $\alpha+1$. Let $\bar{c}\in {^{<\mu}M}$. Lemma
  (\ref{lemma_self_classical_H}) shows that
  $E^{\alpha+1}_{M,\bar{a}}(M,\bar{a}\bar{c},\alpha) = 1$.  So
  $E^{\alpha+1}_{M,\bar{a}}(M,\bar{a}\bar{c},\alpha) =
  E^{\alpha+1}_{N,\bar{b}}(M,\bar{a}\bar{c},\alpha)$ if and only if
  there is some $\bar{d}\in {^{\lh(\bar{c})}N}$ such that
  $E^\alpha_{M,\bar{a}\bar{c}} = E^\alpha_{N,\bar{b}\bar{d}}$, which
  holds, by induction, if and only if there is some $\bar{d}\in
  {^{\lh(\bar{c})}N}$ such that $(M,\bar{a}\bar{c})\sim_\alpha
  (N,\bar{b}\bar{d})$. Similarly, if $\bar{d}\in {^{<\mu}N}$ the lemma
  shows that $E^{\alpha+1}_{N,\bar{b}}(N,\bar{b}\bar{d},\alpha)=1$. So
  $E^{\alpha+1}_{N,\bar{b}}(N,\bar{b}\bar{d},\alpha) =
  E^{\alpha+1}_{M,\bar{a}}(N,\bar{b}\bar{d},\alpha)$ if and only if
  there is some $\bar{c}\in {^{\lh(\bar{d})}M}$ such that
  $E^\alpha_{N,\bar{b}\bar{d}} = E^\alpha_{M,\bar{a}\bar{c}}$, which
  holds, by induction, if and only if there is some $c\in
  {^{\lh(\bar{d})}M}$ such that $ (N,\bar{b}\bar{d}) \sim_\alpha
  (M,\bar{a}\bar{c})$. Thus, if $E^{\alpha+1}_{M,\bar{a}} =
  E^{\alpha+1}_{N,\bar{b}}$ then $(M,\bar{a})\sim_{\alpha+1}
  (N,\bar{b})$.

  For the converse, recall $(M,\bar{a})\sim_{\alpha+1} (N,\bar{b})$ if
  and only if
  \begin{eqn}
    \begin{split} & \hbox{for all }\bar{c}\in {^{<\mu}M}\hbox{ there is some }\bar{d}\in {^{\lh(\bar{c})}N} \hbox{ such that } (M,\bar{a}\bar{c}) \sim_\alpha (N,\bar{b}\bar{d}) \\
     & \hbox{for all }\bar{d}\in {^{<\mu}N}\hbox{ there is some }\bar{c}\in {^{\lh(\bar{d})}M} \hbox{ such that } (M,\bar{a}\bar{c}) \sim_\alpha (N,\bar{b}\bar{d}).
  \end{split}\end{eqn}

  By the inductive hypothesis, this holds if and only if 
  \begin{eqn}
    \begin{split} & \hbox{for all }\bar{c}\in {^{<\mu}M}\hbox{ there is some }\bar{d}\in {^{\lh(\bar{c})}N} \hbox{ such that } E^\alpha_{M,\bar{a}c} = E^\alpha_{N,\bar{b}d} \hbox{ and}\\
      & \hbox{for all }\bar{d}\in {^{<\mu}N}\hbox{ there is some }\bar{c}\in {^{\lh(\bar{d})}M}\hbox{ such that } E^\alpha_{M,\bar{a}c} = E^\alpha_{N,\bar{b}d}.
    \end{split}\end{eqn}

  So suppose $(K,\bar{s}\bar{t},\alpha)\in \dom(E^{\alpha+1}_{M,\bar{a}})$.

  The reader should observe a real difference raises its head in here, in that
  $\dom(G^{\alpha+1}_{M,N,\bar{a}}) \subsetneq \dom(H^{\alpha+1}_{M,\bar{a}})$. 
  However, we are still able to treat the two cases notationally identically.

  If  $E^{\alpha+1}_{M,\bar{a}}(K,\bar{s}\bar{t},\alpha) =1$ there is
  is some $\bar{c}\in {^{\lh(\bar{t})}M}$ such that
  $E^\alpha_{M,\bar{a}\bar{c}} = E^\alpha_{K,\bar{s}\bar{t}}$.

  Since $(M,\bar{a})\sim_{\alpha+1} (N,\bar{b})$ there is some $\bar{d}\in {^{\lh(\bar{t})}N}$
  such that $(M,\bar{a}\bar{c})\sim_{\alpha} (N,\bar{b}\bar{d})$. By induction,
  $E^\alpha_{M,\bar{a}\bar{c}} = E^\alpha_{N,\bar{b}\bar{d}}$
  Hence $E^\alpha_{N,\bar{b}\bar{d}} =
  E^\alpha_{K,\bar{s}\bar{t}}$ and so $E^{\alpha+1}_{N,\bar{c}}(K,\bar{s}\bar{t},\alpha) =1$.

  Similarly, if
  $E^{\alpha+1}_{N,\bar{d}}(K,\bar{s}\bar{t},\alpha) =1$ there is some
  $\bar{c}\in {^{\lh(\bar{t})}M}$ such that
  $E^\alpha_{M,\bar{a}\bar{c}} = E^\alpha_{K,\bar{s}\bar{t}}$ and $E^{\alpha+1}_{M,\bar{a}}(K,\bar{s}\bar{t},\alpha) =1$.

  Thus $E^{\alpha+1}_{M,\bar{a}}(K,\bar{s}\bar{t},\alpha) =1$ if and only
  if $E^{\alpha+1}_{N,\bar{c}}(K,\bar{s}\bar{t},\alpha) =1$, as required.
\end{proof}

At this point we can continue with an analogue of Scott's analysis and prove a Scott isomorphism theorem.

\begin{proposition} Let $M$ be an $L$-structure. There is an ordinal $\alpha<\card{M}^+$ such that if 
$\bar{a}$, $\bar{b}\in {^{<\lambda}M}$ with $\lh(\bar{a})=\lh(\bar{b})$ and $(M,\bar{a})\sim_\alpha (M,\bar{b})$ then for all
$\beta\in \On$ we have $(M,\bar{a})\sim_\beta (M,\bar{b})$. The least such $\alpha$ is the \emph{Scott rank} of $M$.
\end{proposition}

\begin{proof} As in \cite{Marker}, Lemma (2.4.14).
\end{proof}

\begin{proposition}\label{Scott_iso_thm} Let $M$, $N$ be countable $L$-structures and $\alpha$ the Scott rank of $M$. Then
  $N\isom M$ if and only if $H^\alpha_{M,\emptyset} = H^\alpha_{N,\emptyset}$ and for all $\bar{a}\in {^{<\lambda}M}$ and $\bar{b}\in {^{\lh(\bar{a})}N}$
  if $H^\alpha_{M,\bar{a}} = H^\alpha_{N,\bar{b}}$ then $H^{\alpha+1}_{M,\bar{a}} = H^{\alpha+1}_{N,\bar{b}}$.
\end{proposition}

\begin{proof} As in \cite{Marker}, Theorem (2.4.15), but using Proposition \ref{prop_funs_agree_equiv_sim_classical_FH}
  for equivalences between
  equality of functions and pairs being related under the
  $\sim_\alpha$ in place of \cite{Marker}, Lemma (2.4.13).
\end{proof}

\section{Relationship with the original Scott-Karp analysis}

  One should note that Proposition
  (\ref{prop_funs_agree_equiv_sim_classical_FH}) is not directly
  inherited from the typical structure of Scott sentences, as a
  conjunction of an infinitary conjunction of existential sentences
  and a single univerally quantified infinitary disjunction, but
  instead mimics the structure of the sentences that appear in
  \cite{CS} and \cite{B00}. This is particulary desirable when one
  generalizes from the truth values being limited to
  $2=\set{0,1}=\set{\top,\bot}$ as in \cite{BMPIII}.

However, we can relate the approach of Proposition (\ref{prop_funs_agree_equiv_sim_classical_FH}) to analysis of $L$-structures to
one involving infinitary sentences.

\begin{definition} Let $M$ be an $L$-structure, $\bar{a}\in {^{<\lambda} M}$ and $\alpha\in \On$.
  Set \[ \phi^0_{M,\bar{a}} = \bigwedge_{\substack{\psi\in \UA \\ M\thinks \psi(\bar{a})}} \psi(\bar{v}) \land
  \bigwedge_{\substack{\psi\in \UA \\ M\thinks \neg \psi(\bar{a})}}\neg \psi(\bar{v}), \]
  for limit $\alpha$ let
  $\phi^\alpha_{M,\bar{a}} = \bigwedge_{\beta<\alpha} \phi^\beta_{M,\bar{a}}$ and let
  \[ \phi^{\alpha+1}_{M,\bar{a}} = \phi^\alpha_{M,\bar{a}}
  \sss\sss \wedge\bigwedge_{\substack{K,\bar{s}\bar{t} \\ M\thinks \exists \bar{u} \phi^\alpha_{K,\bar{s}\bar{t}}(\bar{a},\bar{u})}} \kern-24pt \exists \bar{u}\phi^\alpha_{K,\bar{s}\bar{t}}(\bar{v},\bar{u}) \sss\sss \land
  \bigwedge_{\substack{K,\bar{s}\bar{t} \\ M\thinks \neg\exists  \bar{u}\phi^\alpha_{K,\bar{s}\bar{t}}(\bar{a},\bar{u})}}  \kern-24pt  \neg\exists \bar{u}\phi^\alpha_{K,\bar{s}\bar{t}}(\bar{v},\bar{u}) .\]
\end{definition}

\begin{proposition}\label{related_to_3.2.4}  Let $M$ be an $L$-structure, $\bar{a}\in {^{<\lambda}|M|}$ and $\alpha\in \On$. Let $K$ and $N$ be $L$-structures,
  $\bar{b}\in {^{\lh(\bar{a})}|N|}$, $\bar{s}\in {^{\lh(\bar{a})}|K|}$ and $\bar{t}\in {^{<\mu}|K|}$.
  Then $M\thinks \exists \bar{u} \phi^\alpha_{K,\bar{s}\bar{t}}(\bar{a},\bar{u})$ if and only if
  $H^{\alpha+1}_{M,\bar{a}}(K,\bar{s}\bar{t},\alpha) = 1$ and
  $H^\alpha_{M,\bar{a}}=H^\alpha_{N,\bar{b}}$ if and only if $N\thinks \phi^{\alpha+1}_{M,\bar{a}}(\bar{b})$.
\end{proposition}

\begin{proof} By double induction. $M\thinks \exists \bar{u} \phi^\alpha_{K,\bar{s}\bar{t}}$ if and only if
  there is some $\bar{c} {^{\lh(\bar{t})} M}$ such that $M\thinks \phi^\alpha_{K,\bar{s}\bar{t}}(\bar{a}\bar{c})$.
  Applying the induction hypothesis and the definition of $H^{\alpha+1}_{M,\bar{a}}$, we have $M\thinks \phi^\alpha_{K,\bar{s}\bar{t}}(\bar{a}\bar{c})$ if and only if
  $H^{\alpha}_{K,\bar{s}\bar{t}} = H^\alpha_{M,\bar{a}\bar{c}}$ if and only if $H^{\alpha+1}_{M,\bar{a}}(K,\bar{s}\bar{t})=1$, as required.

  On the other hand, $H^{\alpha+1}_{M,\bar{a}}=H^{\alpha+1}_{N,\bar{b}}$ if and only if
    for every $K$,  $\bar{s}\in {^{\lh(\bar{a})}K}$ and $\bar{t}\in {^{\lh(\bar{a}+1)}K}$
 we have $M\thinks \exists \bar{u} \phi^\alpha_{K,\bar{s}\bar{t}}(\bar{a},\bar{u})$ if and only if 
 $N\thinks \exists \bar{u} \phi^\alpha_{K,\bar{s}\bar{t}}(\bar{b},\bar{u})$ if and only if $N\thinks \phi^{\alpha+1}_{M,\bar{a}}(\bar{b})$.
\end{proof}

\section{Degrees and ranks}

We introduce the notion of \emph{quantifier degree} which we will use in the ensuing.

\begin{definition}\label{quantifier_degree} (Quantifier degree) Let $\phi \in L_{\infty\lambda}$.
  The \emph{quantifier degree} of $\phi$, $d(\phi)$, is defined by induction on its complexity.
  \begin{center}
\vskip-10pt
\begin{tabular}{ll}
 $\phi$ atomic  & \qquad $d(\phi)= 0$ \\
  $\phi = \neg \psi$ & \qquad $d(\phi)= d(\psi)$ \\
  $\phi = \psi \to \chi$ &  \qquad $d(\phi)= \max(\set{d(\psi), d(\chi)})$ \\
  $\phi = \Wedge \Phi$ & \qquad $d(\phi) = \sup\setof{d(\psi)}{\psi \in \Phi}$ \\
  $\phi = \Vee \Phi$ & \qquad $d(\phi) = \sup\setof{d(\psi)}{\psi \in \Phi}$ \\ 
  $\phi = \forall V \psi$ & \qquad $d(\phi) = d(\psi) + 1$ \\
  $\phi = \exists V \psi$ & \qquad $d(\phi) = d(\psi) + 1$ \\
\end{tabular}
\end{center}
\end{definition} 

\begin{note-nono} When $\lambda=\omega$ this definition induces the `usual' quantifier degree on $\bar{L}_{\infty\omega}$.
  Observe, for example, writing $\exists x $ for $\exists \set{x}$ for $\set{x}$ a singleton,
  if $\phi$ is $\exists \set{x,y}\sss x=y$ and $\psi$ is $\exists x
  \exists y  \sss x=y$ then $d(\phi)=1$, but $d(\psi)=2$.
\end{note-nono}

\begin{definition} Let $\alpha$ be an ordinal.
  $L^{\alpha}_{\infty,\lambda}$ is the collection of $L_{\infty,\lambda}$-formulae $\phi$
  such that $d(\phi)\le \alpha$. We write $\bar{L}^{\alpha}_{\infty,\lambda} $ for
  $\bar{L}_{\infty,\lambda} \cap L^{\alpha}_{\infty,\lambda}$.
\end{definition}

We also introduce a \emph{modified rank}, which takes into account the complexity gained in allowing function symbols in
the language. This rank is a slight variant of that found in
\cite{EFT}, Exercise (XII.3.15),  which in turn is a variant of
a rank found in \cite{Flum}, p.254. It is used in \cite{EFT} to give a framework for a proof of a Fra\"iss\'e-style theorem
for arbitrary finite first order languages (\emph{i.e.},
possibly with function symbols and symbols for constants).

\begin{definition}\label{modified_rank} (Modified Rank) Let $\phi \in L_{\infty\lambda}$.
  The \emph{modified rank} of terms $t$ and formulae $\phi$, $r(t)$, $r(\phi)$, is defined by induction on their complexity.
  \begin{center}
\vskip-10pt
\begin{tabular}{ll}
  $t=v$  & \qquad $r(t)= 0$ \\
  $t = c$ & \qquad $r(t)= 1$ \\
  $t= f(s_0,\dots,s_{n-1})$ & \qquad $r(t)= 1 + r(s_0) + \dots r(s_{n-1})$ \\
  $\phi = R(t_0,\dots,t_{m-1})$ & \qquad $r(\phi)= r(t_0) + \dots r(t_{m-1})$ \\
  $\phi = \hbox{`}t_0 = t_1\hbox{'}$ & \qquad $r(\phi)=  \max\set{0,r(t_0) + r(t_1) -1}$ \\
  $\phi = \neg \psi$ & \qquad $r(\phi)= r(\psi)$ \\
  $\phi = \psi \to \chi$ & \qquad $r(\phi)= \max(\set{r(\psi), r(\chi)})$ \\
  $\phi = \Wedge \Phi$ & \qquad $r(\phi) = \sup\setof{r(\psi)}{\psi \in \Phi}$ \\
  $\phi = \Vee \Phi$ & \qquad $r(\phi) = \sup\setof{r(\psi)}{\psi \in \Phi}$ \\ 
  $\phi = \forall V \psi$ & \qquad $r(\phi) = r(\psi) + 1$ \\
  $\phi = \exists V \psi$ & \qquad $r(\phi) = r(\psi) + 1$ \\
\end{tabular}
\end{center}
\end{definition} 

\begin{remark}\label{m_rank_le_qdeg} Note for any formula $\phi$ we have $d(\phi)\le r(\phi)$.
If the language is relational then we always have equality: $d$ and $r$ are the same.
However in languages without function symbols, but with constants we already have $r\ne d$. 
\end{remark}

Unnested atomic formulae were defined in Definition (\ref{UA_fmlae}).

\begin{definition}\label{defn_nested} An atomic $L$-formula $\phi$ is \emph{nested} if it is not unnested, \emph{i.e.}, it is of one of the following forms:
  \begin{itmz}
  \item $c= d$, where $c$, $d\in {\mathcal C}$ are constant symbols
  \item $f(t_0,\dots,t_{n-1}) =t_n$ or $t_n = f(t_0,\dots,t_{n-1})$, where $f$ is an $n$-ary function symbol for some $n<\omega$ and $t_0,\dots,t_n$ are $L$-terms, not all of which are variable symbols.
    \item $R(t_0,\dots,t_{n-1})$, where $R$ is an $n$-ary relation symbol for some $n<\omega$ and $t_0,\dots,t_{n-1}$ are $L$-terms, not all of which are variable symbols.
  \end{itmz}
\end{definition}

We next show how nested atomic formulae are equivalent to unnested formulae. This material appears in \cite{Hodges},\S{}2.6, ``by example.''
We give a fuller account for completeness. See also, for example, \cite{EFT}, Chapter VIII,\S{}1.

\begin{definition}\label{defn_phi^exists} We recursively define for each atomic $L$-formula $\phi(\bar{v})$,
  with free variables $\bar{v}$, an unnested, in fact positive primitive, formula $\phi^\exists(\bar{v})$.

  For unnested atomic formulae we simply take $\phi^\exists = \phi$.

  Now suppose $\phi$ is nested.

If $\phi$ is $c=d$, where $c$, $d$ are constant symbols we set $\phi^\exists$ to be $\exists v_0 \sss v_0 = c\sss\sss\&\sss\sss v_) = d $.

Now, suppose $\phi$ is $f(t_0,\dots,t_{n-1}) =t_n$, where $f$ is an $n$-ary function symbol for some $n<\omega$ and $t_0,\dots,t_n$ are $L$-terms, not all of which are variable symbols.

Suppose that $S = \setof{i_j}{j\le m}$ is an enumeration of the collection of $i\le n$ such that $t_i$ is not a variable symbol.
For $i\ni S$ let $v_i = t_i$ and let $\setof{v_{i_j}}{j\le m}$ be a collection of distinct variable symbols
none of which appear in any of the $t_i$ for $i<n$. We set $\phi'$ to be
\[\exists v_{i_0}\sss\dots\sss \exists v_{i_{m}} \sss\sss
v_{i_0} = t_{i_0} \sss\sss\&\sss\sss \dots \sss\sss\&\sss\sss v_{i_{m}} = t_{i_{m}} \sss\sss\&\sss\sss   f(v_0,\dots,v_{n-1}) = v_n.\] 
We then set $\phi^\exists$ to be 
\begin{eqn}
  \begin{split} \exists v_{i_0}\sss\dots\sss \exists v_{i_{m}} \sss\sss (v_{i_0} = t_{i_0}& )^{\exists} \sss\sss\&\sss\sss \dots
    \sss\sss\&\sss\sss (v_{i_{m}} = t_{i_{m}})^\exists\\
    &
    \sss\sss\&\sss\sss f(v'_0,\dots,v'_{n-1}) = v'_n
  \end{split}
\end{eqn}
where for $i\ni S$ we have $v'_i = v_i$ and for $i=i_j\in S$ we have $v'_i = v_{i_j}$.

The definition for $\phi$ being $t_n = f(t_0,\dots,t_{n-1})$ is almost identical, we just switch the order of the final conjunct to be $ v_n = f(v_0,\dots,v_{n-1})$.

Finally, suppose $\phi$ is $R(t_0,\dots,t_{n-1})$, where $R$ is an $n$-ary relation symbol for some $n<\omega$ and $t_0,\dots,t_{n-1}$ are $L$-terms, not all of which are variable symbols.
We have a similar definition to the previous one.

Suppose that $S = \setof{i_j}{j\le m}$ is an enumeration of the collection of $i< n$ such that $t_i$ is not a variable symbol.
For $i\ni S$ let $v_i = t_i$ and let $\setof{v_{i_j}}{j\le m}$ be a collection of distinct variable symbols
none of which appear in any of the $t_i$ for $i<n$. 
We then set $\phi^\exists$ to be 
\begin{eqn}
  \begin{split} \exists v_{i_0}\sss\dots\sss \exists v_{i_{m}} \sss\sss (v_{i_0} = t_{i_0}& )^{\exists} \sss\sss\&\sss\sss \dots
    \sss\sss\&\sss\sss (v_{i_{m}} = t_{i_{m}})^\exists\\
    &
    \sss\sss\&\sss\sss R(v'_0,\dots,v'_{n-1}),
  \end{split}
\end{eqn}
where for $i\ni S$ we have $v'_i = v_i$ and for $i=i_j\in S$ we have $v'_i = v_{i_j}$.
\end{definition}

\begin{definition} If $\psi$ is an arbitrary $L$-formula let $\psi^\exists$ be the formula derived from $\psi$ by for each atomic subformula $\phi$ replacing $\phi$ by $\phi^\exists$.
\end{definition}

\begin{proposition}\label{rank=degree_of_ext} If $\psi$ is an arbitrary $L$-formula then $r(\psi)=d(\psi^\exists)$.
\end{proposition}

\begin{proof} Immediate from the definitions.
\end{proof}

\begin{theorem}
  Let $\phi(\bar{v})$ be an atomic formula with $\bar{v} = \tup{v_0, \dots, v_{n-1}}$. Let $M$ be an $L$-structure and $\bar{a} =\tup{a_0, \dots, a_{n-1}} \in |M|$.
  Then $M\thinks\phi(\bar{a})$ if and only if $M\thinks \phi^\exists(\bar{a})$.
\end{theorem}

\begin{proof}
It is immediate the theorem holds for the unnested atomic formulae since for these formulae $\phi^\exists =\phi$.. 
We  show remainder of the theorem by induction on complexity of $L$-terms appearing in atomic formulae.

Let $\phi(\bar v)$ be ``$c=d$'' with $c,d \in |C|$ constant symbols. 
Let $\mathfrak{c}^M$ and $\mathfrak{d}^M$ be the interpretations of $c$ and $d$ in $M$.

Then $M\thinks (c=d)(\bar{a})$ if and only if $M\thinks \mathfrak{c}^M = \mathfrak{d}^M$.
However $M\thinks \mathfrak{c}^M = \mathfrak{d}^M$ if and only if
$M\thinks \exists x (x = \mathfrak{c}^M \sss\sss\&\sss\sss x = \mathfrak{d}^M)$,
and the latter holds if and only if $M\thinks \exists x (x = c \sss\sss\&\sss\sss x = d)(\bar{a})$,
that is, $M\thinks (c=d)^\exists(\bar{a})$.

The proofs in the cases where $\phi(\bar{v})$ is ``$ \tau_m (\bar{v}) =f(\tau_0(\bar{v}), \dots, \tau_{m-1}(\bar{v}))$''
and for $j \le m$ we have $\tau_j (\bar{v})$ are $L$-terms which are
not all variable symbols and $f$ is an $m$-ary function symbol, and where $\phi(\bar{v})$ is
``$R(\tau_0(\bar{v}), \dots, \tau_{m-1}(\bar{v}))$'' and
for $j < m$ the $\tau_j (\bar{v})$ are $L$-terms, not being all variable symbols and $R$ is an $m$-ary relation symbol are similar. 
\end{proof}

\begin{corollary} If $\psi$ is an $L$-formula, $M$ is an $L$-structure
  and $\bar{a}\in {^n|M|}$ then $M\thinks \psi(\bar{a})$ if and only if
  $M\thinks\psi^\exists(\bar{a})$.
\end{corollary}

\begin{proof} By the inductive definition of the interpretation of instantiations of formulae.
\end{proof}

\section{Karp's theorem}\label{Karp_theorem}

We give an alternative proof of a result of Karp (\cite{Karp})
generalizing a result of Fra\"iss\'e (\cite{Fra56}) for finite $\alpha$.

\begin{definition}\label{defn_equiv_alpha} Let $M$, $N$ be $L$-structures and $\alpha\in\On$. Then $M\equiv_\alpha N$ if
  $M$ and $N$ satisfy the same sentences of modified rank $\le \alpha$.
\end{definition}

\begin{theorem}\label{invariance_of_fml_level_by_level} 
  Let $\alpha$ be an ordinal. Let $M$, $N$ be $L$-structures,
  $\bar{a}\in {^{<\lambda}M}$, $\bar{b}\in {^{\lh(\bar{a})}N}$ and
  $h:\bar{a}\longrightarrow \bar{b}$ such that for all
  $i<\lh(\bar{a})$ we have $h(a_i)=b_i$.

  Then $h\in Q_\alpha$ if and only if 
  for all unnested $L$-formulae
  $\phi(\bar{v})$ with $d(\phi) = \alpha$, we have that $\phi$ is
  invariant under $h$: $M\thinks \phi(\bar{a})$ if and only if
  $N\thinks \phi(\bar{b})$.
  
  Consequently, by Proposition (\ref{rank=degree_of_ext}), $(M,\bar{a}) \sim_\alpha (N,\bar{b})$ if and only
  for all $L$-formulae
  $\psi(\bar{v})$ with $r(\psi) = \alpha$ we have $M\thinks \psi(\bar{a})$ if and only if
  $N\thinks \psi(\bar{b})$.
\end{theorem}

\begin{proof} See \cite{Vaananen}, Theorem (9.27), (3) implies (1), which reuses the proof of
  \cite{Vaananen}, Theorem (7.47), (ii) implies (i), for the left-to-right direction. See the proof
  of \cite{Vaananen}, Theorem (7.47), (i) implies (ii), for the right-to-left direction, making the modification
  of replacing ``finite'' by ``of size $<\lambda$'', taking extensions of size $<\mu$ rather than singletons
  and institing that the functions in the $P_\beta$ must extend $h$. 
\end{proof}

\begin{corollary}   Let $M$, $N$ be $L$-structures and let $\alpha\in\On$. Then
  $M\sim_\alpha N$ if and only if $M\equiv_\alpha N$.
\end{corollary}

\vskip12pt

\begin{theorem}\label{Karp_thm} (Karp's theorem) Suppose $M$ and $N$ are $L$-structures.
  Then $M\equiv_\alpha N$ if and only if $F^\alpha_{M,\emptyset} = F^\alpha_{N,\emptyset}$.
\end{theorem}

\begin{proof} Combine Theorem (\ref{invariance_of_fml_level_by_level}) and
  Proposition (Proposition (\ref{prop_funs_agree_equiv_sim_classical_FH}).
\end{proof}

\vskip40pt

\section{Hjorth analysis and games}\label{Hjorth_intro_games}

One advantage of our ``formula-free'' approach to Scott-Karp analysis
is that it is relatively flexible. Consequently, in this section we
are able to give a version of our previous results for Hjorth's
generalized analysis, \cite{Hjorth}, \cite{Drucker-arxiv},
\cite{Drucker}, which takes place in the context of actions of
topological groups. In this context it is not at all clear how to give
a version of Scott-Karp analysis in the traditional style of one model
satisfying a formula derived from another model.


We start by giving a version of Drucker's version (\cite{Drucker-arxiv}) of Hjorth's work.
Hjorth (and Drucker) work with Polish group
actions, where the topology on the group has a countable basis. This
is important for some, descriptive set theoretic, aspects of their
work where there is a need for effectivity. Here we work without this
restriction.

Throughout this section, let $X$ be a topological space and $G$ a topological group acting continuously on $X$.

\begin{notation}Write $\mathcal T^*$ for the collection of non-empty, open subsets of $G$.
\end{notation}

\begin{notation} For $U\in \mathcal T^*$ let $U'\pret U$ abbreviate 
  $U'\in \mathcal T^* \sss\sss\&\sss\sss U'\subseteq U$. We frequently employ this convention
  to simplify quantified expressions. So, if $\phi$ is
  some mathematical formula we write $\exists U'\pret U\sss\phi $ for
  $\exists U'\in \mathcal T^* ( U'\subseteq U \sss\sss\&\sss\sss \phi)$.
\end{notation}

We define by induction a series of asymmetric relations on pairs from $X\times {\mathcal T}^*$ indexed by the ordinals.

\begin{definition} For $U$, $V \in {\mathcal T}^*$ and $x$, $y\in X$ we define
  \begin{eqn}\begin{split} & \qquad\qquad (x,U) \le_0  (y,V) \hbox{ if } \obar{U\cdot x} \subseteq \obar{V\cdot y} ,\\
      &\hbox{for every ordinal }\alpha\hbox{, } (x,U) \le_{\alpha+1} (y,V) \hbox{ if }\\
      &\qquad\qquad \forall U'\pret U \sss \sss \exists V' \pret V \sss\sss (y,V') \le_{\alpha}  (x,U'),\\
      &\hbox{for every limit ordinal }\lambda\hbox{, }(x,U) \le_{\lambda} (y,V) \hbox{ if }\\
      &
      \qquad\qquad  \forall \alpha<\lambda\sss\sss (x,U) \le_{\alpha} (y,V).
    \end{split}
  \end{eqn}
\end{definition}

\begin{lemma}\label{le_forms_chain} Let $\alpha\in \On$\vskip-36pt
  \begin{enumerate}[align=parleft,labelsep=1cm, label=(\roman*)] \item The relation $\le_\alpha$ is transitive.
\item  Let $U$, $U'$, $V$, $V' \in {\mathcal T}^*$ with $U'\subseteq U$ and $V\subseteq V'$, and let $x$, $y\in X$.
  Then $ (x,U) \le_{\alpha} (y,V) \hbox{ implies } (x,U') \le_{\alpha} (y,V')$.
\item If $\beta< \alpha$ and $(x,U) \le_{\alpha} (y,V) $ then $(x,U) \le_{\beta} (y,V)$
  \end{enumerate}
  \end{lemma}

\begin{proof} See \cite{Drucker-arxiv}, Lemma (3.2) and Lemma (3.4). (The only non-immediate item is (c).) 
\end{proof}

 We follow classical practice (as recounted in, for example,
 \cite{Harrison-Trainor}) and define equivalence relations on pairs
 from $X\times \mathcal T^*$, indexed by the ordinals, in two
 ways. One way is by symmetrizing the $\le_\alpha$, the other by
 intrinsic definition.

  \begin{definition} For $\alpha \in \On$, $(x,U)$, $(y,V)\in X\times \mathcal T^*$ define
    \[ (x,U) \approx_{\alpha} (y,V) \hbox{ if } (x,U) \le_{\alpha} (y,V) \sss\sss\&\sss\sss (y,V) \le_{\alpha} (x,U). \]
  \end{definition}

  (\cite{Harrison-Trainor}) uses the notation $\equiv_\alpha$ for symmetrization, however we already used this symbol for
  equivalence up to rank $\alpha$, so we use $\approx_\alpha$ here.)
  
\begin{definition} For $(x,U)$, $(y,V)\in X\times \mathcal T^*$ define
  \begin{eqn}\begin{split} & \qquad\qquad (x,U) \sim_0  (y,V) \hbox{ if } \obar{U\cdot x} = \obar{V\cdot y} ,\\
      &\hbox{for every ordinal }\alpha\hbox{, }  (x,U) \sim_{\alpha+1} (y,V) \hbox{ if } (x,U) \sim_{\alpha} (y,V) \hbox{ and} \\
      &\qquad\qquad \forall U'\pret U \sss\sss \exists V' \pret V \sss\sss (y,V') \sim_{\alpha}  (x,U') \hbox{, and}\\
      &\qquad\qquad \forall V'\pret V \sss\sss \exists U' \pret U \sss\sss (x,U') \sim_{\alpha}  (y,V') ,\\
      &\hbox{for every limit ordinal }\lambda\hbox{, }(x,U) \sim_{\lambda} (y,V) \hbox{ if }\\
      &
      \qquad\qquad  \forall \alpha<\lambda\sss\sss (x,U) \sim_{\alpha} (y,V).
    \end{split}
  \end{eqn}
\end{definition}

\begin{lemma} Let $\alpha\in \On$\vskip-12pt
  \begin{enumerate}[align=parleft,labelsep=1cm, label=(\roman*)]
      \item The relation $\approx_\alpha$ is transitive.
\item If $\beta< \alpha$ and $(x,U) \approx_{\alpha} (y,V) $ then $(x,U) \approx_{\beta} (y,V)$
\item The relation $\sim_\alpha$ is transitive.
  \item If $\beta< \alpha$ and $(x,U) \sim_{\alpha} (y,V) $ then $(x,U) \sim_{\beta} (y,V)$
  \end{enumerate}
  \end{lemma}

\begin{proof} Immediate for the most part, using also Lemma (\ref{le_forms_chain}(iii)) for (ii).
\end{proof}


We next include in some observations about games and the relations we have defined, analogous to
those for classical model theory which can be found, for example, in
\cite{Vaananen},\cite{Rothmaler}, \cite{Marker} and \cite{Marker-inf}.

\begin{definition} Let $(x,U)$, $(y,V)\in X\times \mathcal T^*$  and $\alpha\in\On$, $0<\alpha$.
  Define a two-player (dynamic Ehrenfeucht-Fra\"iss\'e-style) game, ${\mathcal G}^s(x,U,y,V,\alpha)$ as follows.
  
  Formally define $z_{-1} = x$, $A_{-1} = U$, $B_{-1}=V$ and $\alpha_{-1}=\alpha$.

  In round $i$, player $I$ plays a triple $(z_i,A_i,\alpha_i)$ and player $II$ responds with $B_i$.

  The players must obey the following rules (or otherwise lose immediately). 

  Player $I$ must choose $\alpha_i < \alpha_{i-1}$ and $z_i \in \set{x,y}$. 

  Furthermore, player $I$ and player $II$ must ensure that if
  $z_i=z_{i-1}$ then $A_i \pret A_{i-1}$ and $B_i\pret B_{i-1}$ and if
  $z_i \ne z_{i-1}$ then $A_i\pret B_{i-1}$ and $B_i\pret A_{i-1}$.

  Player $II$ must also ensure,
  setting $z_i^\dagger$ to be the unique element of $\set{x,y}\setminus \set{z_i}$,
  \[      \obar{B_i\cdot z_i^\dagger}  = \obar{A_i\cdot z_i}  \]
 \end{definition}

\begin{definition} Let $(x,U)$, $(y,V)\in X\times \mathcal T^*$  and $\alpha\in\On$, $0<\alpha$.
  Define a two-player (dynamic Ehrenfeucht-Fra\"iss\'e-style) game ${\mathcal G}^a(x,U,y,V,\alpha)$ as follows.
  
  Formally, for $i\in \omega$ define $z_{2i} = x$, $z^\dagger_{2i} =y $ and $z_{2i-1}=y$, $z^\dagger_{2i-1}=x$.
  Also, let $A_{-1} = U$, $B_{-1}=V$ and $\alpha_{-1}=\alpha$.

  In round $i$, player $I$ plays a pair $(A_i,\alpha_i)$ and player $II$ responds with $B_i$. 

  The players must obey the following rules (or otherwise lose immediately). 

  Player $I$ must choose $\alpha_i < \alpha_{i-1}$ and ensure that
  $A_i \pret B_{i-1}$. Player $II$ must ensure $B_i\pret A_{i-1}$ and
  that $ \obar{B_i\cdot z_i^\dagger} \subseteq \obar{A_i\cdot z_i} . $

 \end{definition}
\vskip12pt

\begin{lemma}\label{win_strats_presereved_at_limits} Let $(x,U)$, $(y,V)\in X\times \mathcal T^*$.
  If $\alpha\in\On$ is a limit ordinal then
  \begin{itemize}
  \item player $II$ has a winning strategy for ${\mathcal
    G}^s(x,U,y,V,\alpha)$ if and only if they have a winning strategy
    for ${\mathcal G}^s(x,U,y,V,\beta)$ for all $\beta\in
    (0,\alpha)$.\vskip6pt
    \item player $II$ has a winning strategy for ${\mathcal G}^a(x,U,y,V,\alpha)$ if and only if
    they have a winning strategy for ${\mathcal G}^a(x,U,y,V,\beta)$ for all $\beta\in (0,\alpha)$. 
  \end{itemize}

\end{lemma}

\begin{proof} For any first move  $(z_0,A_0,\alpha_0)$ by player $I$,
  player $II$ can respond by playing their winning strategy for
  ${\mathcal G}^s(x,U,y,V,\alpha_0+1)$ (respectively ${\mathcal
    G}^a(x,U,y,V,\alpha_0+1)$), since $\alpha$ is a limit ordinal and
  so $\alpha_0<\alpha$ implies $\alpha_0+1<\alpha$.
\end{proof}

\begin{proposition}\label{sim_iff_ws} Let $(x,U)$, $(y,V)\in X\times \mathcal T^*$. Suppose $(x,U)\sim_0 (y,V) $.
  Then for all $\alpha\in \On$ with $0<\alpha$ we have  $(x,U)\sim_\alpha (y,V) $
 if and only if  Player $II$ has a winning strategy for ${\mathcal  G}^s(x,U,y,V,\alpha)$.
\end{proposition}

\begin{proof} Let $\alpha\in \On$ with $0<\alpha$.
  
Suppose  $(x,U)\sim_\alpha (y,V) $. We define by induction on $\omega$ a winning strategy
$\sigma=\tupof{\sigma_i}{i\in \omega}$ for $II$ such that
for each $j\in(0, \omega)$ if for all $i<j$ player $I$ has played $(z_i,A_i,\alpha_i)$ and player II has followed $\sigma_i$ then
$(z_{j-1},A_{j-1})\sim_{\alpha_{j-1}} (z^\dagger_{j-1},\sigma_{j-1}( \tupof{(z_i,A_i,\alpha_i)}{i< j}))$.

Suppose we have defined $\sigma_i$ for $i<j$ and now wish to define $\sigma_j$. Suppose player $I$ plays $(z_j,A_j,\alpha_j)$.
We have $\alpha_j < \alpha_{j-1}$ so $(z_{j-1},A_{j-1})\sim_{\alpha_{j}+1} (z^\dagger_{j-1},B_{j-1})$.

If $z_j = z_{j-1}$ then $A_j \pret A_{j-1}$ and by the definition of $\sim_{\alpha_j +1}$ there is some $B\pret B_{j-1}$ such that
$(z_{j},A_{j})\sim_{\alpha_{j}} (z^\dagger_{j},B)$. Set $\sigma_{j}( \tupof{(z_i,A_i,\alpha_i)}{i\le j}) = B$.

If $z_j \ne z_{j-1}$ then $A_j \pret B_{j-1}$ and by the definition of $\sim_{\alpha_j +1}$ there is some $B\pret A_{j-1}$ such that
$(z_{j},A_{j})\sim_{\alpha_{j}} (z^\dagger_{j},B)$. Set $\sigma_{j}( \tupof{(z_i,A_i,\alpha_i)}{i\le j}) = B$. 

This is, clearly, a winning strategy for $II$.

Conversely, suppose $\sigma=\tupof{\sigma_i}{i\in \omega}$ is a winning strategy for player $II$. We work by induction on the ordinals.
The inductive hypothesis for $\alpha\in\On$ is that the proposition holds for all $\beta\in (0,\alpha)$, that is for all
$(x,U)$, $(y,V)\in X\times \mathcal T^*$ we have $(x,W)\sim_\beta (y,Z) $
if and only if  Player $II$ has a winning strategy for ${\mathcal  G}^s(x,W,y,Z,\beta)$.

We address the case of successor $\alpha+1$. Suppose $U'\pret U$. Consider the play of the game in which the first move for player $I$ is
$(x,U',\alpha)$. Let $\sigma_0(\tup{(x,U',\alpha)})=B$, where $B\pret V$. Since $\sigma$ is a winning strategy for player $II$ in
${\mathcal  G}^s(x,U,y,V,\alpha+1)$ and $\sigma_0(\tup{(x,U',\alpha)})=B$
we have that $\tupof{\sigma_i}{0<i<\alpha+1}$ is a winning strategy for player $II$ in ${\mathcal  G}^s(x,U',y,B,\alpha)$.
Thus, by the inductive hypothesis, $(x,U')\sim_\alpha (y,B)$. Similarly, if $V'\pret V$, consider the play of the game in which the first move for
player $I$ is $(y,V')$. Let $\sigma_0(\tup{(y,V',\alpha)})=B$, where this time $B\pret U$. Again, since $\sigma$ is a winning strategy for player $II$ in
${\mathcal  G}^s(x,U,y,V,\alpha+1)$ and $\sigma_0(\tup{(y,V',\alpha)})=B$
we have that $\tupof{\sigma_i}{0<i<\alpha+1}$ is a winning strategy for player $II$ in ${\mathcal  G}^s(y,B,x,V,\alpha)$.
Thus, by the inductive hypothesis, $(y,V')\sim_\alpha (x,B)$. Hence $(x,U)\sim_{\alpha+1} (y,V)$.
(Note that this argument also covers the base case where $\alpha+1 =1$.)

By Lemma (\ref{win_strats_presereved_at_limits}), if $\alpha$ is a
limit ordinal, player $II$ has a a winning strategy in ${\mathcal
  G}^s(x,U,y,V,\alpha)$ if and only if for all $\beta\in (0,\alpha)$
player $II$ has a winning strategy in ${\mathcal
  G}^s(x,U,y,V,\beta)$. By the inductive hypothesis the latter is so
if and only if for all $\beta\in (0,\alpha)$ we have $(x,U)\sim_\beta
(y,V)$. By definition, this gives $(x,U)\sim_\alpha (y,V)$.
\end{proof}

\begin{proposition}\label{le_iff_ws} Let $(x,U)$, $(y,V)\in X\times \mathcal T^*$. Suppose $(x,U)\le_0 (y,V) $.
  Then for all $\alpha\in \On$ with $0<\alpha$ we have  $(x,U)\le_\alpha (y,V) $
 if and only if  Player $II$ has a winning strategy for ${\mathcal  G}^a(x,U,y,V,\alpha)$.
\end{proposition}

\begin{proof} The proof is very similar to the proof of Proposition (\ref{sim_iff_ws}).

Let $\alpha\in \On$ with $0<\alpha$.
  
Suppose  $(x,U)\le_\alpha (y,V) $. We define by induction on $\omega$ a winning strategy
$\sigma=\tupof{\sigma_i}{i\in \omega}$ for $II$ such that
for each $j\in(0, \omega)$ if for all $i<j$ player $I$ has played $(z_i,A_i,\alpha_i)$ and player II has followed $\sigma_i$ then
\[ (z^\dagger_{j-1},\sigma_{j-1}( \tupof{(z_i,A_i,\alpha_i)}{i< j})) \le_{\alpha_{j-1}} (z_{j-1},A_{j-1}) .\]

Suppose we have defined $\sigma_i$ for $i<j$ and now wish to define $\sigma_j$. Suppose player $I$ plays $(z_j,A_j,\alpha_j)$.
We have $\alpha_j < \alpha_{j-1}$ so $(z^\dagger_{j-1},B_{j-1}) \le_{\alpha_{j}+1} (z_{j-1},A_{j-1})$.

By the definition of $\le_{\alpha_j +1}$ there is some $B\pret A_{j-1}$ such that
$(z^\dagger_{j},B) \le_{\alpha_{j}} (z_{j},A_{j}) $. Set $\sigma_{j}( \tupof{(z_i,A_i,\alpha_i)}{i\le j}) = B$. 

This is, clearly, a winning strategy for $II$.

Conversely, suppose $\sigma=\tupof{\sigma_i}{i\in \omega}$ is a
winning strategy for player $II$. We work by induction on the
ordinals.  The inductive hypothesis for $\alpha\in\On$ is that the
proposition holds for all $\beta\in (0,\alpha)$, that is for all
$(x,U)$, $(y,V)\in X\times \mathcal T^*$ we have $(x,W)\le_\beta (y,Z)
$ if and only if Player $II$ has a winning strategy for ${\mathcal
  G}^a(x,W,y,Z,\beta)$.

We address the case of successor $\alpha+1$. Suppose $U'\pret U$. Consider the play of the game in which the first move for player $I$ is
$(x,U',\alpha)$. Let $\sigma_0(\tup{(x,U',\alpha)})=B$, where $B\pret V$. Since $\sigma$ is a winning strategy for player $II$ in
${\mathcal  G}^s(x,U,y,V,\alpha+1)$ and $\sigma_0(\tup{(x,U',\alpha)})=B$
we have that $\tupof{\sigma_i}{0<i<\alpha+1}$ is a winning strategy for player $II$ in ${\mathcal  G}^s(y,B,x,U',\alpha)$.
Thus, by the inductive hypothesis, $(y,B)\le_\alpha (x,U')$. Hence $(x,U)\le_{\alpha+1} (y,V)$.
(Note that this argument also covers the base case where $\alpha+1 =1$.)

By Lemma (\ref{win_strats_presereved_at_limits}), if $\alpha$ is a limit ordinal, player $II$ has a a winning strategy in
${\mathcal G}^s(x,U,y,V,\alpha)$ if and only if for all $\beta\in (0,\alpha)$ player $II$ has a winning strategy in ${\mathcal G}^s(x,U,y,V,\beta)$. By the inductive hypothesis
the latter is so if and only if for all $\beta\in (0,\alpha)$ we have $(x,U)\sim_\beta (y,V)$. By definition, this gives $(x,U)\sim_\alpha (y,V)$.

\end{proof}

  The proofs of Proposition (\ref{sim_iff_ws}) and Proposition
  (\ref{le_iff_ws}) are similar to the proof of \cite{Vaananen},
  Proposition (7.17), on substituting the notions of game and
  $\sim_\alpha$ defined here for those of a game and $\simeq^\alpha_p$
  in \cite{Vaananen}. However there are differences arising from
  the fact that we argue with the relations $\sim_\alpha$ and
  $\le_\alpha$ directly rather than defining back-and-forth sequences
  as V\"a\"an\"anen does.

    One could define a notion of back-and-forth sequence here. The
    constituents of the sets in the sequence would be functions with
    domains a decreasing sequence of elements of $\mathcal
    T^*$. However, it was not clear how to prove analogues of
    \cite{Vaananen}, Proposition (7.17), without insisting that the
    function $\tup{(U,V)}$ was an element of each set in the
    back-and-forth sequence, and this led to proofs similar to those
    of Proposition (\ref{sim_iff_ws}) and Proposition
    (\ref{le_iff_ws}) with additional notational baggage and no
    further ideas.

    \begin{corollary}  Let $(x,U)$, $(y,V)\in X\times \mathcal T^*$. Suppose $(x,U)\approx_0 (y,V) $.
  Then for all $\alpha\in \On$ with $0<\alpha$ we have  $(x,U)\approx_\alpha (y,V) $
 if and only if  Player $II$ has a winning strategy for ${\mathcal  G}^a(x,U,y,V,\alpha)$ and a winning strategy in ${\mathcal  G}^a(y,V,x,U,\alpha)$.
    \end{corollary}

\begin{definition} Let $(x,U)$, $(y,V)\in X\times \mathcal T^*$. Suppose $(x,U)\approx_0 (y,V) $.
  Then for all $\alpha\in \On$ with $0<\alpha$ let ${\mathcal  G}^z(x,U,y,V,\alpha)$ be the game where
  player $I$ first chooses $(x,U)$ or $(y,V)$ and then the players play a run of ${\mathcal  G}^a(x,U,y,V,\alpha)$
  if player $I$ chose $(x,U)$ and play a run of ${\mathcal  G}^a(y,V,x,U,\alpha)$ if player $I$ chose $(y,V)$.
\end{definition}

Loosely speaking, ${\mathcal  G}^z(x,U,y,V,\alpha)$ is a coproduct of the games ${\mathcal  G}^a(x,U,y,V,\alpha)$
and ${\mathcal  G}^a(y,V,x,U,\alpha)$.

    \begin{corollary}  Let $(x,U)$, $(y,V)\in X\times \mathcal T^*$. Suppose $(x,U)\approx_0 (y,V) $.
  We have  $(x,U)\approx_\alpha (y,V) $
  if and only if  Player $II$ has a winning strategy for ${\mathcal  G}^z(x,U,y,V,\alpha)$.
    \end{corollary}


\section{Hjorth analysis, a Karp-style theorem}\label{Hjorth_Karp_thm}

  We will now define functions similar to those defined in \S{}\ref{Scott_analysis_wo_fmlae} and prove similar results.
  
      \begin{definition} For $(x,U)\in X\times \mathcal T^*$ set $\dom(I_{x,U})=1$ and $I_{x,U}(0) = \obar{U\cdot x}$. 
      \end{definition}
      \begin{definition}
          We inductively define functions $F^\alpha_{x,U}$, $G^{\alpha}_{x,y,U}$ and $H^\alpha_{x,U}$.
  For $\beta<\alpha$ we will have
  \[ F^{\beta}_{x,U} \subseteq F^{\alpha}_{x,U} ,\sss\sss H^{\beta}_{x,U} \subseteq H^{\alpha}_{x,U}  \hbox{ and }
  G^{\beta}_{x,y,U}\subseteq G^{\alpha}_{x,y,U}  .\]

  Let
  \[ F^0_{x,U} = G^0_{x,y,U} = H^0_{x,U} = I_{x,U}. \]
  
  Elements of the domain of $F^{\alpha}_{x,U}$ will either be elements
  of $1$ or a pair whose second entry is an ordinal less than
  $\alpha$.  Element of the domains of the $G^\alpha_{x,y,U}$ and
  $H^{\alpha}_{x,U}$ wither be elements of $1$ or triples whose last
  entry is an ordinal less than $\alpha$.  For $0<\beta<\alpha$ we
  will have
  \[ F^{\alpha}_{x,U}\on \beta = F^{\beta}_{x,U},\sss\sss H^{\alpha}_{x,U}\on \beta = H^{\beta}_{x,U}
                        \hbox{ and }  G^{\alpha}_{x,y,U}\on \beta = G^{\beta}_{x,y,U} .\]
    Thus, in each case, for limit $\alpha$ the functions are
    completely determined by the definitions for
    $\beta<\alpha$. Furthermore, for $\alpha+1$ and tuples in their
    domain with last element $\beta<\alpha$ the functions are
    determined by the definitions for $\alpha$.

    Finally, we treat for $\alpha+1$ tuples in the functions domains with last element $\alpha$.

    We let $H^{\alpha+1}_{x,U}(y,V,\alpha) = 1$ if there is some
    $U'\in \mathcal T^*$ with $U'\subseteq U$ such that
    $H^{\alpha}_{x,U'} = H^{\alpha}_{y,V}$, and $=0$ otherwise.

    Similarly, we let $F^{\alpha+1}_{x,U}(F^\alpha_{y,V},\alpha) = 1$
    if there is some $U'\in \mathcal T^*$ with $U'\subseteq U$ such
    that $F^{\alpha}_{x,U'} = F^{\alpha}_{y,V}$, and $=0$ otherwise.

    Elements of the domain of $G^{\alpha+1}_{x,y,U}$ with last element
    $\alpha$ are triples of the form $(z,W,\alpha)$ with $z \in
    \set{x,y}$. Set $z^\dagger$ to be the unique element of $\set{x,y}
    \setminus \set{z}$. $G^{\alpha+1}_{x,y,U} = 1$ if there is some
    $U'\in \mathcal T^*$ with $U'\subseteq U$ such that
    $G^\alpha_{x,y,U'} = G^\alpha_{z,z^\dagger,W}$.
      \end{definition}

      Similarly to \S{}\ref{Scott_analysis_wo_fmlae} we have the
      following two lemmas and proposition. The proofs are essentially
      identical to those of Lemma (\ref{reln_Fs_and_Hs}), Lemma
      (\ref{lemma_self_classical_H}) and Proposition
      (\ref{prop_funs_agree_equiv_sim_classical_FH}) with only the base
      case of the induction differing.
      
\begin{lemma}\label{reln_Fs_and_Hs_Hjorth} Let $(x,U)$, $(y,V)$, $(z,W)\in X\times\mathcal T^*$ and $\beta<\alpha\in\On$.
       Then
       \[ F^{\alpha}_{x,U}(F^\alpha_{y,V},\alpha) = H^{\alpha}_{x,U}(y,V,\alpha) .\]
  Moreover, 
  \[ F^{\alpha}_{x,U} = F^\alpha_{z,W} \hbox{ if and only if } H^{\alpha}_{x,U} = H^{\alpha}_{z,W}.\]
\end{lemma}

      \begin{lemma}\label{lemma_self_Hjorth} Let $x$, $y \in X$, $U$, $U'\in \mathcal T^*$ with $U'\subseteq U$ and let
        $0<\beta<\alpha\in\On$
        \[ F^{\alpha}_{x,U}(F^{\beta}_{x,U'},\beta) = G^{\alpha}_{x,y,U}(x,U',\beta) = H^{\alpha}_{x,U}(x,U',\beta) =1 .\]
\end{lemma}

This gives us the following proposition, which seems to be the closest
we can get to a classical Karp/Fra\"iss\'e-style theorem in this
setting.
      
\begin{proposition}\label{prop_funs_agree_equiv_sim_Hjorth} Let $x$, $y \in X$, $U$, $V\in \mathcal T^*$ and $\alpha\in \On$.
  Then $H^\alpha_{x,U} = H^\alpha_{y,V}$ if and only if  $(x,U)\sim_\alpha (y,V)$, and for $0<\alpha$, if and only if
  Player $II$ has a winning strategy for ${\mathcal  G}^s(x,U,y,V,\alpha)$.
\end{proposition}


\end{document}